\documentclass{amsart}
\usepackage{amsfonts}
\usepackage{color}
\usepackage{appendix}
\def\R{{\mathbb {R}}}
\def\N{{\mathbb {N}}}

\def\A{{\mathcal{A}}}

\def\K{{\mathcal{K}}}
\def\F{{\mathcal{F}}}

\def\lam{\lambda}
\def\vp{\varphi}

\def\ve{\varepsilon}
\def\cd{\rightharpoonup}

\def\cp{\operatorname {\text{cap}}}

\newtheorem{teo}{Theorem}[section]
\newtheorem{lema}[teo]{Lemma}
\newtheorem{prop}[teo]{Proposition}
\newtheorem{corol}[teo]{Corollary}

\theoremstyle{remark}
\newtheorem{remark}[teo]{Remark}

\theoremstyle{definition}
\newtheorem{defi}[teo]{Definition}

\numberwithin{equation}{section}

\begin{document}
	
\title{Optimal partition problems for the fractional laplacian}
\author[A. Ritorto]{Antonella Ritorto}

\address{Departamento de Matem\'atica, FCEN -- Universidad de Buenos Aires and IMAS -- CONICET, Buenos Aires, Argentina}

\email[A. Ritorto]{aritorto@dm.uba.ar}
	
	

\begin{abstract}
In this work, we prove an existence result for an optimal partition problem of the form
$$
\min \{F_s(A_1,\dots,A_m)\colon A_i \in \A_s, \, A_i\cap A_j =\emptyset \mbox{ for } i\neq j\},
$$
where $F_s$ is a cost functional with suitable assumptions of monotonicity and lowersemicontinuity, $\A_s$ is the class of admissible domains and the condition $A_i\cap A_j =\emptyset$ is understood in the sense of the Gagliardo $s$-capacity, where $0<s<1$. Examples of this type of problem are related to the fractional eigenvalues. In addition, we prove some type of convergence of the $s$-minimizers to the minimizer of the problem with $s=1$, studied in \cite{Bucur-Buttazzo-Henrot}. 
\end{abstract}
	
\maketitle


\section{Introduction}

Let $\Omega$ be an open bounded subset of $\R^n$. Fix  $0<s<1$ and $m\in \N$. We consider optimal partition problems of the form
\begin{equation}\label{prob}
\min\left\{ F_s(A_1, \dots, A_m)\colon A_i \in \A_s(\Omega), A_i\cap A_j = \emptyset \mbox{ for } i\neq j \right\},
\end{equation}
where $F_s$ is a cost functional which satisfies some lower semicontinuity and monotonicity assumptions and $\A_s(\Omega)$ denotes the class of admissible domains. 

Optimal partition problems were studied by several authors: Bucur, Buttazzo and Henrot \cite{Bucur-Buttazzo-Henrot}, Bucur and Velichkov \cite{Bucur-Velichkov}, Caffarelli and Lin \cite{Caffarelli-Lin}, Conti, Terracini and Verzini \cite{Conti-Terracini-Verzini-03,Conti-Terracini-Verzini}, Helffer,  Hoffmann-Ostenhof and Terracini \cite{Helffer-Hoffmann-Ostenhof-Terracini}, among others.





In \cite{Caffarelli-Lin}, Caffarelli and Lin established the existence of classical solutions to an optimal partition problem for the Dirichlet eigenvalue, as well as the regularity of free interfaces. One more recent work about regularity of solutions to optimal partition problems  involving eigenvalues of the Laplacian is \cite{Ramos-Tavares-Terracini}, where Ramos, Tavares and Terracini used the existence result of \cite{Bucur-Buttazzo-Henrot} and proved that the free boundary of the optimal partition is locally a $C^{1,\alpha}$-hypersurface up to a residual set.

 
 Conti, Terracini and Verzini proved in \cite{Conti-Terracini-Verzini-03} the existence of the minimal partition for a problem in N-dimensional
 domains related to the method of nonlinear eigenvalues introduced by Nehari in \cite{Nehari}. Moreover, they showed some connections between
 the variational problem and the behavior of competing species systems with large
 interaction. 
 
 Tavares and Terracini proved in \cite{Tavares-Terracini} the existence of infinitely many sign-changing solutions for the system of $m$-Schr\"odinger equations with competition interactions and the relation between the energies associeted and an optimal partition problem which involves $m$-eigenvalues of the Laplacian operator. 

In a recent work  \cite{BRS}, we studied a general shape optimization problem where $m=1$.

For more references related to optimal partition problems see, for instance, \cite{Bonnaillie-Noel-Llena, Bozorgnia, Bucur-Buttazzo, Buttazzo, Conti-Terracini-Verzini,Helffer-Hoffmann-Ostenhof,Osting-White-Oudet, Snelson}

\medskip

The goal of this article is to prove the existence of an optimal partition for the problem \eqref{prob}, where $F_s$ is decreasing in each coordinate and lower semicontinuous for a suitable notion of convergence in $\A_s(\Omega)$, which is the set of admissible domains. This existence result is carried out in Section \ref{sec.teo1}. The dependence on $s$ is related to the Gagliardo $s$-capacity measure and the fractional Laplacian operador $(-\Delta)^s$, we will detail that and other preliminares in Section \ref{sec.preli}.

We follow the ideas given by Bucur, Buttazzo and Henrot in \cite{Bucur-Buttazzo-Henrot}, where was proved the existence of solution to \eqref{prob} in the case $s=1$. 

Furthermore, we prove convergence of the minima and the optimal partition shapes to those of the case $s=1$, studied in \cite{Bucur-Buttazzo-Henrot}. This last aim is accomplished in Section \ref{sec.teo2}.

At the end of this work, we include an Appendix with useful properties of $s$-capacity. Most of those results, we suppose are well-known. Despite of that, we decided to incorporate them for completeness.

\section{Preliminaries and statements} \label{sec.preli}
\subsection{Notations and preliminaries}
Given $s\in(0,1)$ we consider the fractional laplacian, that for smooth functions $u$ is defined as
\begin{align*}
(-\Delta)^s u(x) &:= c(n,s)\mbox{p.v.}\int_{\R^n} \frac{u(x)-u(y)}{|x-y|^{n+2s}} \, dy\\
&=-\frac{c(n,s)}{2}\int_{\R^n} \frac{u(x+z) - 2u(x) + u(x-z)}{|z|^{n+2s}}\, dz. 
\end{align*}
where $c(n,s):= ( \int_{\R^n} \frac{1-\cos\zeta_1}{|\zeta|^{n+2s}}d\zeta )^{-1}$ is a normalization constant. 

The constant $c(n,s)$ is chosen in such a way that the following identity holds,
$$
(-\Delta)^s u = \F^{-1}(|\xi|^{2s}\F(u)),
$$
for $u$ in the Schwarz class of rapidly decreasing and infinitely differentiable functions, where $\F$ denotes the Fourier transform. See \cite[Proposition 3.3]{DiNezza-Palatucci-Valdinoci}.


The natural functional setting for this operator is the fractional Sobolev space $H^s(\R^n)$ defined as
\begin{align*}
H^s(\R^n)&:=\left\{u\in L^2(\R^n) \colon \frac{u(x)-u(y)}{|x-y|^{\frac{n}{2}+s}}\in L^2(\R^n \times \R^n) \right\}\\
&= \left\{ u\in L^2(\R^n)\colon |\xi|^2\F(u)\in L^2(\R^n)\right\}
\end{align*}
which is a Banach space endowed with the norm $\|u\|^2_s:= \|u\|_2^2 + [u]^2_s $,
where the term
$$
	[u]^2_s:=\iint_{\R^n \times \R^n} {\frac{|u(x)-u(y)|^2}{|x-y|^{n+2s}} \, dxdy} 
$$
is the so-called Gagliardo semi-norm of $u$. 

To comtemple the {\em boundary} condition, we work in $H_0^s(\Omega)$, which is the closure of $C_c^\infty(\Omega)$ in the norm $\| \cdot \|_s$. When $\Omega$ is a Lipschitz domain, $H^s_0(\Omega)$ coincides with the space of functions vanishing outside $\Omega$, i.e., 
$$
H^s_0(\Omega)=\{u\in H^s(\R^n)\colon u=0 \mbox{ in } \R^n\setminus \Omega\}.
$$

From now on,  $\Omega\subset \R^n$ will be a Lipschitz domain.

\begin{defi} Given $A\subset \Omega$, for any $0<s<1$, we define the Gagliardo $s-$capacity of $A$ relative to $\Omega$ as
$$
\cp_s(A,\Omega)= \inf \left\{ [u]^2_s \colon u\in C^\infty_c(\Omega),\ u\ge 1 \mbox{ in a neighborhood of } A \right\}.
$$

We say that a subset $A$ of $\Omega$ is an {\em $s$-quasi open} subset of $\Omega$ if there exists a decreasing sequence $\{G_k\}_{k\in \N}$ of open sets such that $\lim_{k\to\infty}\cp_s(G_k,\Omega)=0$ and $A\cup G_k$ is an open set. 

We denote by $\A_s(\Omega)$ the class of all $s-$quasi open subsets of $\Omega$. 

In the case $s=1$ the definitions are completely analogous with $\|\nabla u\|_2^2$ instead of $[u]_s^2$.
\end{defi}

We say that a property $P(x)$ holds {\em $s$-quasi everywhere} on $E\subset \Omega$ ( $s$-q.e. on $E$), if $\cp_s(\{x\in E\colon P(x) \mbox{ does not hold} \}, \Omega)=0$.


A function $u\colon \R^n \to \R$ is said {\em $s$-quasi-continuous} if there exists a decreasing sequence $\{G_k\}_{k\in \N}$ of open sets such that $\lim_{k\to\infty}\cp_s(G_k, \Omega)= 0$ and $u|_{\R^n \setminus G_k}$ is continuous.

The following theorem allows us to work with $s$-quasi continuous functions instead of the classical fractional Sobolev ones. 

\begin{teo}[Theorem 3.7, \cite{Warma}]\label{representante}
For every function $u\in H^s_0(\Omega)$ there exist a unique $\tilde{u}\colon \R^n \to \R$ $s$-quasi-continuous function such that $u=\tilde{u}$ a.e. in $\R^n$.
\end{teo}

From this point, we identify a function $u\in H_0^s(\Omega)$ with its $s$-quasi continuous representative.

For $A \in \A_s(\Omega)$, we consider the fractional Sobolev space 
$$
H_0^s(A):= \{ u \in H_0^s(\Omega) \colon u = 0 \ \mbox{$s$-q.e. in } \R^n \setminus A \}.
$$

To go into detail about $s$-capacity we refer the reader, for instance, to \cite{Shi-Xiao,Warma}.

\medskip

\subsection{Statements}
Given $A\in \A_s(\Omega)$, we denote by $u^s_A\in H^s_0(A)$ the unique weak 
solution to
\begin{equation}\label{uas}
(-\Delta)^s u^s_A =1 \quad \mbox{ in } A, \qquad u^s_A=0 \quad \mbox{ in }  \R^n \setminus A.
\end{equation}
With this notation, we define the following notion of set convergence.

\begin{defi}[Strong $\gamma_s$-convergence]
Let $\{A_k\}_{k\in\N}\subset \A_s(\Omega)$ and $A\in \A_s(\Omega)$. We say that $A_k\stackrel{\gamma_s}{\to} A$ if $u_{A_k}^s\to u_A^s$ strongly in $L^2(\Omega)$. 

Let $m\in \N$, $\{(A_1^k,\dots,A_m^k)\}_{k\in \N} \subset \A_s(\Omega)^m$ and $(A_1,\dots,A_m) \in \A_s(\Omega)^m$. We say $(A_1^k,\dots, A_m^k)\stackrel{\gamma_s}{\to}(A_1,\dots,A_m)$ if $A_i^k \stackrel{\gamma_s}{\to} A_i $ for every $i=1, \dots, m$.
\end{defi}

\begin{defi}[Weak $\gamma_s$-convergence]
Let $\{A_k\}_{k\in\N}\subset \A_s(\Omega)$. We say that $A_k\stackrel{\gamma_s}{\cd} A$ if $u_{A_k}^s\to u$ strongly in $L^2(\Omega)$ and $\{u>0\}=A$.
	
Let $m\in\N$ and $\{(A_1^k,\dots,A_m^k)\}_{k\in \N} \subset \A_s(\Omega)^m$. We say $(A_1^k,\dots, A_m^k)\stackrel{\gamma_s}{\cd}(A_1,\dots,A_m)$ if $A_i^k \stackrel{\gamma_s}{\cd} A_i
$ for every $i=1, \dots, m$.
\end{defi}

\medskip 

Let $m\in \N$ be fixed and $0<s\le1$. Let $F_s \colon \A_s(\Omega)^m\to [0,\infty]$ be such that
\begin{itemize}
\item $F_s$ is weak $\gamma_s$-lower semicontinuous, that is, 
$$
F_s(A_1, \dots, A_m)\le \liminf_{k\to \infty} F_s(A_1^k, \dots, A_m^k),
$$
for every sequence $\{(A_1^k,\dots, A_m^k)\}_{k\in \N}$ such that $(A_1^k,\dots, A_m^k) \stackrel{\gamma_s}{\cd} (A_1,\dots, A_m)$.
\item $F_s$ is decreasing, that is, for every $(A_1,\dots,A_m), (B_1,\dots,B_m)  \in \A_s(\Omega)^m$ such that $A_i\subset B_i$ for $i=1,\dots, m$, we have
$$
F_s(A_1, \dots, A_m) \ge F_s(B_1, \dots, B_m).
$$ 
\end{itemize}	

Under these assumptions, we are able to recover the existence result of \cite{Bucur-Buttazzo-Henrot}, for the fractional case. Rigorously speaking, we have the following theorem. 

\begin{teo}\label{main.s}
Let $F_s \colon \A_s(\Omega)^m\to [0,\infty]$ be a decreasing and weak $\gamma_s$-lower semicontinuous functional. Then, there exists a solution to
\begin{equation}\label{min.s}
\min \left\{  F_s(A_1, \dots, A_m) \colon A_i \in \A_s(\Omega), \ \cp_s(A_i\cap A_j,\Omega )=0  \mbox{ for } i\neq j    \right\}.
\end{equation} 
\end{teo}

The proof of Theorem \ref{main.s} is carried out in Section \ref{sec.teo1} and we use ideas from \cite{Bucur-Buttazzo-Henrot} and \cite{BRS}.

\medskip

Once we know the existence of an optimal partition shape for each $0<s<1$, we want to analyze the limit of these minimizers and its minimum values when $s\uparrow 1$. To this aim, we need a suitable relationship between the cost functionals $F_s, 0<s\le 1$ and a notion of set convergence.

Let us start with the notion of set convergence. For $A\in \A_1(\Omega)$, we introduce the analogous notation $u_A^1 \in H_0^1(A)$ for the unique weak solution to 
$$
-\Delta u_A^1 =1 \mbox{ in } A, \quad u_A^1=0 \mbox{ in } \R^n \setminus A.
$$

\begin{defi}[$\gamma$-convergence] Let $0<s_k\uparrow 1$, $\{A_k\}_{k\in\N}\subset \A_{s_k}(\Omega)$ and $A\in \A_1(\Omega)$. We say that $A_k\stackrel{\gamma}{\to} A$ if $u_{A_k}^{s_k}\to u_A^1$ strongly in $L^2(\Omega)$. 
	
Let $m \in \N$, $(A_1^k,\dots,A_m^k)\in \A_{s_k}(\Omega)^m$ and $(A_1,\dots,A_m)\in \A_1(\Omega)^m$. We say that $(A_1^k,\dots,A_m^k) \stackrel{\gamma}{\to}(A_1,\dots,A_m)$ if $u_{A_i^k}^{s_k} \to u_{A_i}^1$ strongly in $L^2(\Omega)$, for every $i=1,\dots, m$. 
\end{defi}

%

Let $m\in \N$ and $0<s\le1$. Let $F_s\colon \A_s(\Omega)^m\to [0,\infty]$ be decreasing and weak $\gamma_s$-lower semicontinuous functionals. Then, there exists $(A_1^s,\dots, A_m^s)$ solution to 
\begin{equation}\label{m.s}
m_s:= \min \left\{ F_s(B_1,\dots,B_m)\colon B_i \in \A_s(\Omega), \, \cp_s(B_i\cap B_j, \Omega)=0 \mbox{ for } i\neq j \right\}.
\end{equation}	

The case $s=1$ was solved in \cite{Bucur-Buttazzo-Henrot}. For $0<s<1$, apply Theorem \ref{main.s}.

Assume the following hypotheses over the cost functionals:
\begin{itemize}
\item[$(H_1)$] Continuity. For every $(A_1,\dots, A_m)\in \A_1(\Omega)^m$, 
$$
F_1(A_1,\dots, A_m)=\lim_{s\uparrow 1}F_s(A_1,\dots, A_m).
$$
\item[$(H_2)$] Liminf inequality. For every $0<s_k\uparrow 1$,  $(A_1^k,\dots, A_m^k)\in \A_{s_k}(\Omega)^m$ and $(A_1,\dots, A_m)\in \A_1(\Omega)^m$ such that $(A_1^k,\dots,A_m^k) \stackrel{\gamma}{\to}(A_1,\dots,A_m)$,
$$
F_1(A_1,\dots,A_m) \le \liminf_{k\to \infty}{F_{s_k}(A_1^k,\dots,A_m^k)}.
$$
\end{itemize}

These conditions $(H_1)-(H_2)$ are natural and analogous to those consider in \cite{BRS}, where a similar shape optimazation problem was studied with $m=1$.

Now, we are able to establish the main result. 

\begin{teo}\label{main}
Let $m\in \N$ be fixed and $0<s\le 1$. Let $F_s\colon \A_{s}(\Omega)^m\to [0,\infty]$ be a decreasing and weak $\gamma_{s}$-lower semicontinuous functional, and such that $(H_1)-(H_2)$ are verified. Then, 
\begin{equation}\label{limite}
m_1 = \lim_{s\uparrow 1} m_s,
\end{equation}
where $m_s$ is defined in \eqref{m.s}.
	
Moreover, if $(A_1^s,\dots,A_m^s)$ is a minimizer of \eqref{m.s}, then, there exist a subsequence $0<s_k\uparrow 1$, $(\tilde{A}_1^{s_k},\dots,\tilde{A}_m^{s_k}) \in \A_{s_k}(\Omega)^m$  and $(A_1^1,\dots,A_m^1) \in \A_1(\Omega)^m$ such that $\tilde{A}_i^{s_k} \supset A_i^{s_k}$ and 
$$
(\tilde{A}_1^{s_k},\dots,\tilde{A}_m^{s_k}) \stackrel{\gamma}{\to}  (A_1^1,\dots,A_m^1) ,
$$
where $(A_1^1,\dots,A_m^1) $ is a minimizer of \eqref{m.s} with $s=1$.
\end{teo}

The proof of Theorem \ref{main} is carried out in Section \ref{sec.teo2} and we use again ideas from \cite{BRS}.

\subsection{Examples}
Given $A\in \A_s(\Omega)$, consider the problem
\begin{equation} \label{eq}
(-\Delta)^s u= \lambda^s u \quad \textrm{ in } A, \qquad u\in H_0^s(A)
\end{equation}
where $\lam^s\in \R$ is the eigenvalue parameter. It is well-known that there exists a discrete sequence $\{\lam_k^s(A)\}_{k\in\N}$ of positive eigenvalues of \eqref{eq} approaching $+\infty$ whose corresponding eigenfunctions $\{u_k^s\}_{k\in\N}$ form an orthogonal basis in $L^2(A)$. Moreover, the following variational characterization holds for the eigenvalues
\begin{equation} \label{variac}
\lam_k^s(A)=\min_{u\perp W_{k-1} }\frac{c(n,s)}{2}\frac{[u]^2_s}{\|u\|_2^2},
\end{equation}
where $W_k$ is the space spanned by the first $k$ eigenfunctions $u_1^s,\ldots, u_k^s$.

Consider functionals $F_s(A_1,\dots,A_m)= \Phi_s(\lam_{k_1}^s(A_1),\dots, \lam_{k_m}^s(A_m))$. Theorem \ref{main.s} claims that for every $(k_1,\dots,k_m)\in\N^m$, the minimum
$$
\min\{ \Phi_s(\lam_{k_1}^s(A_1),\dots, \lam_{k_m}^s(A_m))\colon A_i\in \A_s(\Omega), \, \cp_s(A_i\cap A_j,\Omega) \mbox{ for } i\neq j\}
$$
is achieved, where $\Phi_s\colon  \R^m\to \bar \R$, is increasing in each coordinate and lower semicontinuous.

Moreover, if $\Phi_s(t_1,\dots,t_m)\to \Phi_1(t_1,\dots,t_m)$ for every $(t_1,\dots,t_m)\in\R^m$ and
$$
\Phi_1(t_1,\dots,t_m)\le \liminf_{k\to\infty} \Phi_{s_k}(t_1^k,\dots,t_m^k),
$$
for every $(t_1^k,\dots,t_m^k)\to (t_1,\dots,t_m)$, then Theorem \ref{main} together with the existence result of \cite{Bucur-Buttazzo-Henrot} imply that 
\begin{align*}
&\min\{ \Phi_1(\lam_{k_1}(A_1),\dots,\lam_{k_m}(A_m))\colon A_i\in \A_1(\Omega), \, \cp_1(A_i\cap A_j,\Omega)=0 \mbox{ for } i\neq j\}\\
& = \lim_{s\uparrow 1} \min\{ \Phi_s(\lam_{k_1}^s(A_1),\dots,\lam_{k_m}^s(A_m))\colon A_i\in \A_s(\Omega), \, \cp_s(A_i\cap A_j,\Omega)=0 \mbox{ for }i\neq j\}.
\end{align*}

\section{Proof of Theorem \ref{main.s}} \label{sec.teo1}

In this section, we adapted the ideas from \cite{Bucur-Buttazzo-Henrot}, where the authors consider the Laplacian operator, to recover their results for the fractional case. Despite the similarity of the proofs, we include them for the reader's convenience and recalling that in the context of this article we need the nonlocal tools proved in the recent work \cite{BRS}.

\subsection{Certain compactness on $\A_s(\Omega)$}

Consider $\K_s$ given by
\begin{equation}\label{Ks}
\K_s:=\{w\in H_0^s(\Omega) \colon w\ge 0, \, (-\Delta)^s w \le 1 \mbox{ in } \Omega \}.
\end{equation}

\begin{prop}[Proposition 3.3 and Lemma 3.5, \cite{BRS}]\label{propKs}
$\K_s$ is convex, closed and bounded in $H_0^s(\Omega)$. Moreover, if $u, v\in \K_s$, then, $\max\{u,v\} \in \K_s$.
\end{prop}

\begin{prop}[Lemma 3.2, \cite{BRS}]\label{uAmax}
Given $A\in \A_s(\Omega)$, $u_A^s$ is  the solution to 
$$
\max \left\{  w\in H_0^s(\Omega) \colon w\le 0 \mbox{ in } \R^n \setminus A, \, (-\Delta)^s w \le 1 \mbox{ in } \Omega   \right\}.
$$

Moreover, $u_A^s \in \K_s$, for every $A\in \A_s(\Omega)$.
\end{prop}

From now on, we undertand the identity $A=\{u_A^s>0\}$ in the sense of the Gagliardo $s$-capacity, thanks to Proposition \ref{propA}.

\begin{remark}\label{wcompacidaddominios}
The weak $\gamma_s$-convergence is sequentially
pre-compact in $\A_s(\Omega)$. Indeed, given a sequence $\{A_k\}_{k\in \N} \subset \A_s(\Omega)$, we know that $\{u_{A_k}^s\}_{k\in \N} \subset \K_s$. By Proposition \ref{propKs}, there exist a subsequence $\{u_{A_{k_j}}^s\}_{j\in \N} \subset \{u_{A_k}^s\}_{k\in \N}$ and a function $u\in \K_s$ such that $u_{A_{k_j}}^s \to u$ strongly in $L^2(\Omega)$. Denote by $A:=\{u>0\}$. Then, $A_{k_j} \stackrel{\gamma_s}{\cd} A$.
\end{remark}

Next proposition allows us to pass from the weak $\gamma_s$-convergence to the strong one, if we are willing to \textit{enlarge} the sequence involved. 
\begin{prop}\label{keyprop}
Let $\{A_k\}_{k\in\N}\subset \A_s(\Omega)$ and $A,B\in \A_s(\Omega)$ be such that  $A_k\stackrel{\gamma_s}{\cd} A\subset B$. 

Then, there exists a subsequence $\{A_{k_j}\}_{j\in\N}\subset \{A_k\}_{k\in \N}$ and a sequence $\{B_{k_j}\}_{j\in \N} \subset \A_s(\Omega)$ such that $A_{k_j}\subset B_{k_j}$ and $B_{k_j}\stackrel{\gamma_s}{\to} B$.
\end{prop}
\begin{proof} Since $A_k\stackrel{\gamma_s}{\cd} A\subset B$, we know that $u_{A_k}^s\to u$ strongly in $L^2(\Omega)$, where $\{ u>0 \}=A$. As a consequence of Propisition \ref{propKs}, $u\in \K_s$. Moreover, by Proposition \ref{uAmax}, $u\le u_A^s$. Since $A\subset B$, $u_A^s\le u_B^s$. Then, $u\le u_B^s$.
	
Denote by $B^{\ve}=\{ u_B^s>\ve \}$ and consider $\{u_{A_{k}\cup B^{\ve}}^s\}_{k\in \N} \subset  \K_s$. Again by Proposition \ref{propKs}, there exists a subsequence $\{A_{k_j}\}_{j\in \N} \subset \{A_k\}_{k\in \N}$ such that $u_{A_{k_j}\cup B^{\ve}}^s \to u^{\ve}$ strongly in $L^2(\Omega)$.

Due to the convergence $u_{A_{k_j}}^s\to u$ strongly in $L^2(\Omega)$ and $u\le u_B^s$, we conclude from \cite[Lemma 3.6]{BRS}, $u^{\ve}\le u_B^s$.
	
Inside the proof of \cite[Lemma 3.7]{BRS}, it was shown that $(u_{B}^s-\ve)^+\le u_{B^{\ve}}^s$. Since $B^{\ve}\subset A_{k_j} \cup B^{\ve}$, it follows that  $u_{B^{\ve}}^s \le u_{A_{k_j} \cup B^{\ve}}^s$. So, taking the limit $j\to\infty$, we obtain
$$
(u_{B}^s-\ve)^+\le u_{B^{\ve}}^s\le u^{\ve}\le u_B^s.
$$

The sequence $\{u^{\ve}\}_{\ve>0}$ is contained in $\K_s$. So, by Proposition \ref{propKs}, up to a subsequence, we know it has a weak limit in $H_0^s(\Omega)$. But, the previous inequality tells that this weak limit should be $u_B^s$. In addition, $u^{\ve}\to u_B^s$ strongly in $L^2(\Omega)$. 

Thus, there exists a sequence $\ve_j\downarrow 0$ such that $u_{A_{k_j}\cup B^{{\ve}_j}}
^s \to u_B^s$ strongly in $L^2(\Omega)$. That is, $A_{k_j}\cup B^{{\ve}_j}=:B_{k_j}\stackrel{\gamma_s}{\to} B$, where 	$\{B_{k_j}\}_{j\in \N}$ is the {\em enlarged sequence}.
\end{proof}

\subsection{An auxiliary functional}

Fix $m \in \N$ and $0<s<1$. Let $F_s\colon \A_s(\Omega)^m \to [0,\infty]$ be a decreasing and $\gamma_s$-lower semicontinuous functional. 

We define a functional $G_s\colon \mathcal{K}_s^m \to [0,\infty]$ 
\begin{equation}\label{G.s}
G_s(w_1,\dots,w_m):=\inf \left\{  \liminf_{k\to\infty}J_s(w_1^k,\dots,w_m^k) \colon w_i^k\to w_i \mbox{ strongly in } L^2(\Omega)    \right\},
\end{equation}
where $J_s\colon \K_s^m \to [0,\infty]$ is defined as
$$
J_s(w_1,\dots,w_m):=\inf \left\{ F_s(A_1,\dots,A_m)\colon A_i\in \A_s(\Omega), \ u_{A_i}^s\le w_i \mbox{ for } i=1,\dots, m   \right\}
$$ 
and $\K_s$ was given in \eqref{Ks}.
 
We will show that $G_s$ satisfies the following properties:
\begin{itemize}
\item[($G_1$)] $G_s$ is decreasing on $\mathcal{K}_s^m$, that is $G_s(u_1,\dots, u_m)\ge G_s(v_1,\dots, v_m)$, if $u_i\le v_i$ for every $i=1,\dots,m$. 
\item[($G_2$)] $G_s$ is lower semicontinuous on $\mathcal{K}_s$ with respect to the strong topology on $L^2(\Omega)$, 
\item[($G_3$)] $G_s(u_{A_1}^s,\dots, u_{A_m}^s)=F_s(A_1,\dots, A_m)$ for every $(A_1,\dots, A_m)\in\A_s(\Omega)^m$.
\end{itemize}	

The conditions $(G_1)$ and $(G_2)$ are easy to check and it is the content of next proposition.

\begin{prop}
With the notation above, $G_s$ satisfies $(G_1)$ and $(G_2)$.
\end{prop}
\begin{proof}
By construction, it is clear that $G_s$ verifies $(G_2)$. 

To prove $(G_1)$, let $(u_1,\dots, u_m), (v_1,\dots, v_m)\in \K_s^m$ such that $u_i\le v_i$ for every $i=1,\dots,m$. 

Take $\{u_i^k\}_{k\in \N} \subset \K_s$ such that $u_i^k\to u_i$ strongly in $L^2(\Omega)$  for every $i=1,\dots,m$ and 
$$
G_s(u_1,\dots, u_m)=\lim_{k\to \infty}J_s(u_1^k,\dots, u_m^k). 
$$

Consider $v_i^k:=\max\{v_i, u_i^k\}$ for every $i=1,\dots,m$ and $k\in \N$. By Proposition \ref{propKs}, we obtain that $v_i^k \in \K_s$. In addition, $v_i^k\to \max\{v_i, u_i\}=v_i$ strongly in $L^2(\Omega)$, for every $i=1,\dots,m$. Thus, noticing that $J_s$ is decreasing, we have
$$
G_s(v_1,\dots,v_m)\le \liminf_{k\to\infty}J_s(v_1^k,\dots,v_m^k)\le \lim_{k\to\infty}J_s(u_1^k,\dots,u_m^k)=G_s(u_1,\dots,u_m).
$$
\end{proof}

Now, we prove the most important property of $G_s$, which is the connection with the cost functional $F_s$.

\begin{prop}
The functional $G_s$ satisfies $(G_3)$.
\end{prop}
\begin{proof}
By definition of $G_s$ \eqref{G.s}, it is clear that $G_s(u_{A_1}^s,\dots,u_{A_m}^s)\le F_s(A_1,\dots, A_m)$, for every $(A_1,\dots, A_m)\in \A_s(\Omega)^m$.

To obtain the other inequality, it is enough to prove that for every sequence $\{u_i^k\}_{k\in\N}\subset \K_s(\Omega)$ such that $u_i^k\to u_{A_i}^s$ strongly in $L^2(\Omega)$ for $i=1,\dots, m$, we have
$$
F_s(A_1,\dots,A_m)\le \liminf_{k\to\infty}J_s(u_1^k,\dots, u_m^k).
$$

By definition of $J_s$, there exists $\{(A_1^k,\dots,A_m^k)\}_{k\in \N} \subset  \A_s(\Omega)^m$ such that
\begin{equation}\label{FsJs}
u_{A_i^k}^s\le u_i^k \mbox{ for } i=1,\dots, m, \, \mbox{ and } \, F_s(A_1^k,\dots, A_m^k)\le J_s(u_1^k,\dots,u_m^k)+\frac{1}{k}.
\end{equation}

By Remark \ref{wcompacidaddominios}, there exists $v_i\in \K_s$ such that $u_{A_i^k}^s\to v_i$ strongly in $L^2(\Omega)$, up to a subsequence. That is, $A_i^k\stackrel{\gamma_s}{\cd}B_i:=\{v_i>0\}\in \A_s(\Omega)$, for every $i=1,\dots,m$. 

Moreover, taking the limit in $u_{A_i^k}^s\le u_i^k$, we obtain that $v_i\le u_{A_i}^s$ for every $i=1,\dots,m$. In addition, we have $B_i\subset A_i=\{u_{A_i}^s>0\}$. We are able to apply Proposition \ref{keyprop}, to obtain the existence of subsequences $\{A_i^{k_j}\}_{j\in \N},\{B_i^{k_j}\}_{j\in \N} \subset \A_s(\Omega)$ such that $A_i^{k_j} \subset B_i^{k_j}$ and $B_i^{k_j}\stackrel{\gamma_s}{\to} A_i$.

Now, using the $\gamma_s$-lower semicontinuity and decreasing property of $F_s$ and \eqref{FsJs}, we conclude
\begin{align*}
F_s(A_1,\dots,A_m)&\le \liminf_{j\to \infty}{F_s(B_1^{k_j},\dots, B_m^{k_j})}\\
&\le \liminf_{j\to \infty}{F_s(A_1^{k_j},\dots, A_m^{k_j})} \\
&\le \liminf_{j\to \infty}{J_s(u_1^{k_j},\dots, u_m^{k_j})},
\end{align*}
which implies the remaining inequality $F_s(A_1,\dots,A_m)\le G_s(u_{A_1}^s,\dots,u_{A_m}^s)$. 
\end{proof}

The decreasing property of a functional $F_s$ makes equivalent its weak and strong $\gamma_s$-lowersemicontinuity, which is the content of next theorem. 

\begin{teo} \label{debileqvfuerte}
Let $m\in \N$ and $0<s<1$. Let $F_s\colon \A_s(\Omega)^m\to[0,\infty]$ be a decreasing functional. Then, the following assertions are equivalent 
\begin{enumerate}
\item $F_s$ is weakly $\gamma_s$-lower semicontinuous.
\item $F_s$ is $\gamma_s$-lower semicontinuous.
\end{enumerate}
\end{teo}
\begin{proof}
 It is enough to prove that the $\gamma_s$-lower semicontinuity implies the weak one. Indeed, consider $\{A_k\}_{k\in \N} \subset \A_s(\Omega)$ such that $A_k \stackrel{\gamma_s}{\to}A\in \A_s(\Omega)$. By Proposition \ref{propA}, $A=\{u_A^s>0\}$. Then, $A_k \stackrel{\gamma_s}{\cd}A\in \A_s(\Omega)$. That proves $(1)\Rightarrow (2)$.

Assume $F_s$ is $\gamma_s$-lower semicontinuous. 

Let $\{A_i^k\}_{k\in\N}\subset \A_s(\Omega)$ such that $A_i^k \stackrel{\gamma_s}{\cd}A_i \in \A_s(\Omega) $, for $i=1,\dots,m$. That is, $u_{A_i^k}^s\to u_i$ strongly in $L^2(\Omega)$ and $A_i=\{u_i>0\}$. 

Since for every $i=1,\dots, m$, $\{u_{A_i^k}^s\}_{k\in \N}\subset \K_s$, by Proposition \ref{propKs}, $u_i \in \K_s$. Moreover, by Proposition \ref{uAmax},  $u_i\le u_{A_i}^s$.

Thus, recalling the funcional $G_s$ defined by \eqref{G.s} and its properties $(G_3),(G_1)$, $(G_2)$ and again $(G_3)$, we conclude
\begin{align*}
F_s(A_1,\dots,A_m)&=G_s(u_{A_1}^s,\dots,u_{A_m}^s)\le G_s(u_{1},\dots,u_{m})\\
&\le \liminf_{k\to \infty} G_s(u_{A_1^k}^s,\dots,u_{A_m^k}^s)\\
&=\liminf_{k\to \infty}F_s(A_1^k,\dots,A_m^k).\\
\end{align*}
That means $F_s$ is weak $\gamma_s$-lower semicontinuous, as we desired.
\end{proof}

\subsection{Existence of an optimal partition}
With the help of the previous outcomes of this section, we are able to prove existence of a minimal partition shape for \eqref{min.s}.

\begin{proof}[Proof of Theorem \ref{main.s}]
Denote by
$$
\alpha:=\inf \left\{ F_s(A_1,\dots,A_m) \colon A_i\in \A_s(\Omega), \cp_s(A_i\cap A_j, \Omega)=0 \mbox{ for } i\neq j \right\}.
$$

Let $\{(A_1^k,\dots,A_m^k)\}_{k\in\N}\subset \A_s(\Omega)^m$ be such that  
$$
\cp_s(A_i^k\cap A_j^k,\Omega)=0 \mbox{ for } i\neq j, \, \mbox{ and } \,  \lim_{k\to \infty}F_s(A_1^k,\dots,A_m^k)=\alpha.
$$

By Remark \ref{wcompacidaddominios}, there exist $A_1\in \A_s(\Omega)$  and a subsequence $\{A_1^{k_j}\}_{j\in \N} \subset \{A_1^k\}_{k\in \N}$ such that $A_1^{k_j}\stackrel{\gamma_s}{\cd}A_1$. Now, consider $\{A_2^{k_j}\}_{j\in \N}$ and apply again Remark \ref{wcompacidaddominios}. Thus, there exist $A_2\in \A_s(\Omega)$ and a subsequence $\{A_2^{k_{j_l}}\}_{l\in \N} \subset  \{A_2^{k_j}\}_{j\in \N}$ such that $A_i^{k_{j_l}} \stackrel{\gamma_s}{\cd} A_i$ for $i=1,2$. Repeating this argument, we find a sequence $\{(A_1^k,\dots, \A_m^k)\}_{k\in \N}$ and $(A_1,\dots,A_m) \in \A_s(\Omega)$ such that $A_i^k\stackrel{\gamma_s}{\cd} A_i$ for every $i=1,\dots,m$.

Since $F_s$ is weak $\gamma_s$-lower semicontinuous, we obtain  
\begin{equation}\label{casi}
F_s(A_1,\dots,A_m) \le \liminf_{k\to\infty}F_s(A_1^k,\dots,A_m^k)=\alpha.
\end{equation}

To finish the proof, let us see $\cp_s(A_i\cap A_j,\Omega)=0$ for $i\neq j$ be satisfied.

Let $i,j \in \{1,\dots, m\}$ be such that $i\neq j$. Notice that this product $u_{A_i^k}^s \cdot u_{A_j^k}^s$ is an $s$-continuous function too, by Lemma \ref{productofunciones}, and $u_{A_i^k}^s \cdot u_{A_j^k}^s=0$ $s$-q.e. in $\R^n \setminus (A_i^k\cap A_j^k)$. Moreover, since $\cp_s(A_i^k \cap A_j^k,\Omega)=0$, we have $u_{A_i^k}^s \cdot u_{A_j^k}^s=0$ $s$-q.e. in $\R^n$. 

By \cite[Lemma 3.8]{Warma}, there exist subsequences $\{u_{A_i^k}^s\}_{k\in \N}$ and $\{u_{A_j^k}^s\}_{k\in \N}$, denoted with the same index, which converge $s$-q.e. to $u_i$ and $u_j$ respectively. Then, passing to the limit, we obtain $u_i \cdot u_j=0$ $s$-q.e. in $\R^n$. That is $\cp_s(\{u_i \cdot u_j \neq 0\},\Omega)=0$. But, $\{u_i \cdot u_j \neq 0\}=A_i \cap A_j$.

We have shown that $(A_1,\dots,A_m)$ is admissible for the minimization problem \eqref{min.s} and recalling \eqref{casi} the result is proved. 
\end{proof}

Due to Theorems \ref{debileqvfuerte} and \ref{main.s}, we can establish the next immediate corollary.	

\begin{corol}\label{coro.s}
Let $F_s \colon \A_s(\Omega)^m\to [0,\infty]$ be a decreasing and $\gamma_s$-lower semicontinuous functional. Then, there exists a solution to \eqref{min.s}.
\end{corol}

\section{Proof of Theorem \ref{main}} \label{sec.teo2}

In this part of the article, we study the behavior of the optimal partition shapes obtained in Section \ref{sec.teo1} and their minimun values. Again, we use some results from \cite{BRS}.

\begin{lema}[Lemma 4.1, \cite{BRS}] \label{Ksacotado}
Let $0<s_k\uparrow 1$ and let $u_k\in \K_{s_k}$. Then, there exists $u\in H^1_0(\Omega)$  and a subsequence $\{u_{k_j}\}_{j\in \N}\subset\{u_k\}_{k\in\N}$ such that $u_{k_j} \to u$ strongly in $L^2(\Omega)$. 
	
Moreover, if $u_k \in \K_{s_k}$ is such that $u_k\to u$ strongly in $L^2(\Omega)$, then $u\in \K_1$.
\end{lema}

Next proposition gives an idea of the limit behavior of $u_A^s$ when the domains also are varying with $s$.

\begin{prop}[Proposition 4.5, \cite{BRS}] \label{suplente}
Let $0<s_k\uparrow1, A^k \in \A_{s_k}(\Omega)$ be such that $u_{A^k}^{s_k}\to u$ strongly in $L^2(\Omega)$. Then, there exist $\tilde{A}^k \in \A_{s_k}(\Omega)$ such that $A^k \subset \tilde{A}^k$ and $\tilde{A}^k$ $\gamma-$converges to $A:=\{ u>0\}$.
\end{prop}

Now we are ready to prove the main result of this article.

\begin{proof}[Proof of Theorem \ref{main}]
First, notifce that $m_1$ is achieved by \cite[Theorem 3.2]{Bucur-Buttazzo-Henrot}.

Let $0<s_k\uparrow1$. By Theorem \ref{main.s}, there exists $(A_1^k,\dots, A_m^k)\in \A_{s_k}(\Omega)^m$ such that
\begin{equation}\label{datosAk}
\cp_{s_k}(A_i^k\cap A_j^k,\Omega)=0 \mbox{ for } i\neq j \,  \mbox{ and } \, F_{s_k}(A_1^k,\dots, A_m^k ) = m_k,
\end{equation}
where $m_k=m_{s_k}$ defined in \eqref{min.s}.

Let $(A_1,\dots, A_m) \in \A_1(\Omega)^m$ be such that $\cp_1(A_i\cap A_j,\Omega)=0$ for $i\neq j$. Since $0<s_k\uparrow 1$, we can assume $0<\ve_0<s_k\uparrow 1$, for some fixed $\ve_0$. 

Now, recalling Corollary \ref{keycorol} and  Remark \ref{A1subsetAs}, we know that $(A_1,\dots,A_m)$ belongs to 
$$
\{(B_1,\dots,B_m)\colon B_i\in \A_{s_k}(\Omega), \, \cp_{s_k}(B_i\cap B_j,\Omega)=0 \mbox{ for } i\neq j \},
$$
for every $k\in \N$. This fact and  condition $(H_1)$ imply that
$$
\limsup_{k\to \infty} F_{s_k}(A_1^k,\dots, A_m^k)\le \lim_{k\to \infty} F_{s_k}(A_1,\dots, A_m) = F_1(A_1,\dots, A_m).
$$
It follows that
\begin{equation}\label{abajo}
\limsup_{k\to\infty} m_k\le m_1.
\end{equation}	

To see the remaining inequality, let us denote $u_i^k:=u_{A_i^k}^{s_k} \in \K_{s_k}$. By Lemma \ref{Ksacotado}, there is $u_i\in \K_1$ such that, up to a subsequence, $u_i^k \to u_i$ strongly in $L^2(\Omega)$ and a.e. in $\Omega$.
 

Denote by $A_i:=\{u_i>0\}\in \A_1(\Omega)$ for every $i=1,\dots, m$. We claim that $\cp_1(A_i \cap A_j,\Omega)=0$ for $i\neq j$. 

Indeed, let $i\neq j$ be fixed.
For each $k\in \N$, due to Lemma \ref{lebesguecap} and \eqref{datosAk}, we know that 
$$
|\{u_i^k \cdot u_j^k\neq 0 \}|=|A_i^k\cap A_j^k|\le C(n,s_k)\cp_{s_k}(A_i^k \cap A_j^k, \Omega)=0.
$$
Then, $u_i^k \cdot u_j^k=0 $ a.e. in $\R^n$.
Since $u_l^k\to u_l$ a.e. in $\Omega$ for $l=1,2$, we conclude  $u_i \cdot u_j =0$ a.e in $\Omega$, it is still true in $\R^n \setminus \Omega$ considering that they belong to $H_0^s(\Omega)$. So, $u_i \cdot u_j=0$ a.e. in $\R^n$.

Reminding that we are working with $1$-quasi continuous representative functions in $H_0^1(\Omega)$, the previous identity  $u_i \cdot u_j=0$ a.e. in $\R^n$ and \cite[Lemma 3.3.30]{Henrot-Pierre} tells that $u_i \cdot u_j =0$ $1$-q.e. in $\R^n$. That means, $\cp_1(A_i\cap A_j, \Omega)=0$. 

Consequently, $(A_1,\dots,A_m)$ is admissible to the problem $\ref{min.s}$ with $s=1$ and we obtain 
$m_1\le F_1(A_1,\dots,A_m)$.
	
Moreover, by Proposition \ref{suplente}, there exists $\tilde{A}_i^k \in \A_{s_k}(\Omega)$ such that $A_i^k \subset \tilde{A}_i^k$ and $(\tilde{A}_1^k, \dots, \tilde{A}_m^k)$  $\gamma-$converges to $(A_1,\dots, A_m)$.

Finally, from  condition $(H_2)$ and the decreasing property of $F_{s_k}$, we   conclude that
\begin{align*}
m_1&\le F_1(A_1,\dots,A_m) \leq \liminf_{k\to \infty} F_{s_k}(\tilde{A}_1^{k}, \dots, \tilde{A}_m^{k}) \\
&\leq \liminf_{k\to \infty} F_{s_k}(A_1^k, \dots, A_m^k)=\liminf_{k\to \infty}m_k.
\end{align*}
Therefore, from the previous conclusion and \eqref{abajo} we have the identity \eqref{limite} and the results follow. 
\end{proof}


\appendix

\section{Some useful properties of $s$-capacity} \label{apendice}

The following lemmas address some basic properties of $s$-capacity and $s$-quasi continuous functions. We suppose those results are well-known and we include them for completeness. 

\begin{lema}\label{productofunciones}
Let $u,v \colon \R^n \to \R$ be $s$-quasi continuous functions. Then, the product $u \cdot v$ is also an $s$-quasi continuous function.  
\end{lema}
\begin{proof}
By definition, there exist decreasing sequences $\{A_k\}_{k\in \N}$ and $\{B_k\}_{k\in \N}$ of open sets such that $\lim_{k\to \infty}\cp_s(A_k,\Omega)=\lim_{k\to \infty} \cp_s(B_k,\Omega)=0$ and $u|_{\R^n \setminus A_k}$, $v|_{\R^n \setminus B_k}$ are continuous. 
	
Consider $C_k:=A_k\cup B_k$. Then, $\{C_k\}_{k\in \N}$ is a decreasing sequence of open sets such that $\lim_{k\to \infty}\cp_s(C_k,\Omega)=0$, since $\cp_s(C_k,\Omega)\le \cp_s(A_k,\Omega)+\cp_s(B_k,\Omega)$ by \cite[Proposition 3.6]{Warma}. Moreover, $(u \cdot v)|_{\R^n \setminus C_k}$ is continuous.
\end{proof}

%

Next lemma gives a relation between the Lebesgue measure and the $s$-capacity of a subset $A \subset \Omega$. The proof is easy and follows \cite[Section 4.7, Theorem 2 VI]{Evans-Gariepy}, where it was shown with the classical capacity measure ($s=1$). 
\begin{lema}\label{lebesguecap}
For every $A\subset \Omega$, $|A|\le C(\Omega,s) \cp_s(A,\Omega),$ where $C(\Omega,s)$ is the Poincar\'e's constant in $H_0^s(\Omega)$.
\end{lema}
\begin{proof}
For every $\ve>0$, there exists a funciton $u_{\ve}\in H_0^s(\Omega)$ such that $u_{\ve}\ge1$ a.e. in a neighborhood of $A$ and
$$
[u_{\ve}]_s^2 \le \cp_s(A,\Omega)+\ve.
$$	
	
On the other hand, by Poincar\'e's inequality, 
$$
|A|=\int_{A}1 \, dx \le \int_{\R^n} u_{\ve}^2 \, dx \le C(\Omega,s)[u_{\ve}]_s^2 \le C(\Omega,s)\left( \cp_s(A,\Omega)+\ve \right).
$$
Take the limit $\ve \downarrow 0$ to obtain the result. 
\end{proof}

For every $A\in \A_s(\Omega)$, we will show that $A=\{u_A^s>0\}$ in the sense of $\cp_s(\cdot, \Omega)$.  To prove this aim, we need some previous results which are modifications from \cite[Lemma 2.1]{DalMaso-Garroni} and \cite[Proposition 5.5]{DalMaso-Murat}.

\begin{lema} \label{A.1}
Let $A\in \A_s(\Omega)$, Then, there exists an increasing sequence $\{v_k\}_{k\in \N}\subset H_0^s(\Omega)$ of non negative functions, such that $\sup_{k\in \N} v_k=1_A$ $s$-q.e. on $\Omega$. 
\end{lema}
We omit the proof since it is completely analogous to that of \cite[Lemma 2.1]{DalMaso-Garroni}.

%
%
%

 We prove a density result in $H_0^s(A)$, for $A\in \A_s(\Omega)$, which is similar to \cite[Proposition 5.5]{DalMaso-Murat}. 

\begin{lema} \label{A.2}
Let $A\in \A_s(\Omega)$. Then, $\{\vp u_A^s \colon \vp \in C_c^{\infty}(\Omega) \}$ is dense in $H_0^s(A)$.
\end{lema}
\begin{proof}
In order to prove the lemma, it is sufficient to see that we can approximate any non negative function  $w\in H_0^s(A)$ with $(-\Delta)^s w\in L^{\infty}(\Omega)$, since $L^{\infty}(\Omega)$ is dense in $H^{-s}(\Omega)$ and $w=w^+ -w^-$. Indeed, for an arbitrary function $w\in H_0^s(\Omega)$, we know that $(-\Delta)^s w=:f \in H^{-s}(\Omega)$. 

Denote by $f:=(-\Delta)^s w$. Then, 
$$
(-\Delta)^s w \le \|f\|_{L^{\infty}(\Omega)} = \|f\|_{L^{\infty}(\Omega)} (-\Delta)^s u_A^s \quad \mbox{ in } A.
$$
By comparison, we obtain  $0\le w\le c u_A^s$, where $c:=\|f\|_{L^{\infty}(\Omega)}$.

For every $\ve>0$, consider $(w-c\ve)^+ \in H_0^s(\Omega)$. Thus,
\begin{equation}\label{inclusion}
\{ (w-c\ve)^+>0\} \subset \{ u_A^s >\ve \}.
\end{equation} 
Notice that $u_A^s \in L^{\infty}(\Omega)$ by \cite[Theorem 4.1]{DiBlasio-Volzone}. Observe that, using \eqref{inclusion}, $\ve<u_A^s \le \|u_A^s\|_{L^{\infty}(\Omega)}$ in $\{(w-c\ve)^+>0\}$. Then, the function $\frac{(w-c\ve)^+}{u_A^s}$ belongs to $H_0^s(\Omega)$.  So, there exists a sequence $\{\vp_k^{\ve}\}_{k\in \N} \subset C_c^{\infty}(\Omega)$ such that $\vp_k^{\ve}  \to \frac{(w-c\ve)^+}{u_A^s}$ strongly in $H_0^s(\Omega)$, when $k\to \infty$. Therefore, $\vp_k^{\ve} u_A^s \to (w-c\ve)^+$ strongly in $H_0^s(\Omega)$, when $k\to \infty$.

On the other hand, $(w-c\ve)^+\to w$ strongly in $H_0^s(\Omega)$, when $\ve\downarrow 0$. 

Consequently, by a diagonal argument, there exist subsequences $\ve_{j}\downarrow 0$ and $\{ \vp_{k_j}^{\ve_{j}}\}_{j\in \N} \subset C_c^{\infty}(\Omega)$ such that $\vp_{k_j}^{\ve_j} u_A^s \to w$ strongly in $H_0^s(\Omega)$. 
\end{proof}

The following proposition is an essential component to relate domains and functions, and  it also contributes to the proofs of the principal results Theorems \ref{main.s} and \ref{main}.

\begin{prop}\label{propA}
Let $A \in \A_s(\Omega)$. Then, $A=\{u_A^s>0\}$ in sense of $\cp_s(\cdot, \Omega)$. That is, $\cp_s(A \triangle \{u_A^s>0\}, \Omega)=0$.
\end{prop}
\begin{proof}
It is clear that $u_A^s=0$ $s$-q.e. on $\R^n \setminus A$. So, $\{u_A^s>0\} \subset A$.
		
To see $A\subset \{u_A^s>0 \}$, we use the previous lemmas. 
	
By Lemma \ref{A.1}, there exists an increasing sequence $\{v_k\}_{k\in \N}\subset H_0^s(\Omega)$ of non negative functions, such that $\sup_{k\in \N} v_k=1_A$ $s$-q.e. on $\Omega$. 
	
For every $v_k$, by Lemma \ref{A.2}, there exists a sequence $\{\vp_{j}^k\}_{j\in \N}\in C_c^{\infty}(\Omega)$ such that $\vp_{j}^k u_A^s \to v_k$ strongly in $H_0^s(\Omega)$ and $s$-q.e., when $j\to \infty$. Since $\vp_{j}^k u_A^s=0 $ $s$-q.e. in $\{u_A^s=0\}$, then $v_k=0$  $s$-q.e. in $\{u_A^s=0\}$. Therefore, $1_A=0$ $s$-q.e. in $\{u_A^s=0\}$, which implies $A\subset \{u_A^s>0\}$. 
\end{proof}

Now, we prove a key estimate used in Section \ref{sec.teo2}, which is a simply remark following the proof of \cite[Proposition 2.2]{DiNezza-Palatucci-Valdinoci}. Notice that we are interested in finding a positive constant connecting in some sense $\cp_s(\cdot, \Omega)$ and $\cp_1(\cdot,\Omega)$. But, we also want that this constant does not depend on $s$. As our goal in Section \ref{sec.teo2} is related to the limit case $s\uparrow 1$, we can assume $0<\ve_0<s<1$ for some $\ve_0$ and that will be enough to obtain this desired and \textit{independent} constant. 

As we said before, the proof of next lemma follows \cite[Proposition 2.2]{DiNezza-Palatucci-Valdinoci} and, despite of the similarity, it is included since we want to analyze how the constant depends on $s$. 

\begin{lema}\label{key}
Let $\ve_0>0$ and $\ve_0 <s <1$. Then, there exits a constant $C>0$ such that for every $u\in H_0^1(\Omega)$
$$
(1-s)[u]_s^2 \le C \|\nabla u\|_{L^2(\Omega)}^2.
$$
and $C=C(\Omega,n,\ve_0)$ does not depend on $s$.
\end{lema}
\begin{proof}
Let $u\in C_c^{\infty}(\Omega)$, we split $[u]_s^2$ into two pieces.
	
For the first part, use the change of variable $z=y-x$ and observe that for $z\in B_1(0)\setminus \{0\}$ and $\vp(t):=u(x+tz)$ for $t\in [0,1]$ we estimate 
$$
\frac{|u(x+z)-u(x)|}{|z|}=\frac{\left|\int_0^1 \vp'(t) \, dt \right|}{|z|}= \frac{\left|\int_0^1 \nabla u(x+tz)\cdot z \, dt \right|}{|z|}\le \int_0^1 | \nabla u(x+tz)|\, dt.
$$
Now, use the previous remark and Jensen's inequality to obtain
\begin{align*}
\int_{\R^n}\int_{\R^n \cap \{|y-x|<1\}} \frac{|u(x)-u(y)|^2}{|x-y|^{n+2s}} \, dx dy &= \int_{\R^n}\int_{B_1(0)} \frac{|u(x)-u(z+x)|^2}{|z|^{n+2s}} \, dz dx\\
&= \int_{\R^n}\int_{B_1(0)} \frac{|u(x)-u(z+x)|^2}{|z|^{2}|z|^{n+2(s-1)}}  \, dz dx\\
&\le \int_{\R^n}\int_{B_1(0)} \left(\int_0^1\frac{|\nabla u(x+tz)|}{|z|^{\frac{n}{2}+s-1}} \, dt\right)^2 dz dx\\
&\le \int_{B_1(0)} \frac{1}{|z|^{n+2(s-1)}} \int_0^1 \|\nabla u\|_{L^2(\Omega)}^2  dt dz \\
&\le \|\nabla u\|_{L^2(\Omega)}^2\int_{B_1(0)} \frac{1}{|z|^{n+2(s-1)}} dz \\
&=\frac{|B_1(0)|}{2(1-s)}\|\nabla u\|_{L^2(\Omega)}.
\end{align*}	
	
For the remaining part, use $|a-b|^2\le 2(a^2+b^2)$ and easily follows 
\begin{align*}
\int_{\R^n}\int_{\R^n \cap \{|y-x|\ge1\}} \frac{|u(x)-u(y)|^2}{|x-y|^{n+2s}} \, dx dy &\le 2\int_{\R^n}\int_{\R^n \cap \{|y-x|\ge1\}} \frac{|u(x)|^2+|u(y)|^2}{|x-y|^{n+2s}} \, dx dy\\
&\le 4 \int_{\R^n}\int_{\R^n \cap \{|y-x|\ge1\}} \frac{|u(x)|^2}{|x-y|^{n+2s}} \, dx dy\\
&\le \int_{\R^n}|u(x)|^2\left( \int_{ \{|z|\ge1\}}  \frac{1}{|z|^{n+2s}} dz\right)\, dx \\
&=\frac{|B_1(0)|}{2s}\| u\|_{L^2(\Omega)} \\
&\le \frac{|B_1(0)|}{2\ve_0}C_1(\Omega,n) \|\nabla u\|_{L^2(\Omega)},
\end{align*}	
where $C_1(\Omega,n)$ is the constant of the classical Poincar\'e's inequality in $H_0^1(\Omega)$.
	
Then, put together the two estimates to conclude
\begin{align*}
(1-s)[u]_s^2 &\le (1-s) \left( \frac{C_1(\Omega,n)}{2\ve_0}+\frac{1}{2(1-s)}\right)|B_1(0)|\|\nabla u\|_{L^2(\Omega)}\\
&\le \left(\frac{C_1(\Omega,n)}{2\ve_0}+\frac{1}{2}\right)|B_1(0)|\|\nabla u\|_{L^2(\Omega)}\\
&= C(\Omega,n,\ve_0) \|\nabla u\|_{L^2(\Omega)}.
\end{align*}
\end{proof}

Automatically, we obtain an estimate relating the $s$-capacity and the $1$-capacity.

\begin{corol}\label{keycorol}
Let $\ve_0>0$ and $\ve_0 <s <1$. Then, there exits a constant $C>0$ such that for every $A\subset \Omega$
$$
(1-s)\cp_s(A,\Omega) \le C \cp_1(A,\Omega),
$$
and $C=C(\Omega,n,\ve_0)$ does not depend on $s$.
\end{corol}

We deduce other useful remark from Lemma \ref{key}: every $1$-quasi open set is also an $s$-quasi open, for $0<s<1$.

\begin{remark}\label{A1subsetAs}
For every $0<s<1$, $\A_1(\Omega)\subset \A_s(\Omega)$. Moreover, if $0<s<t\le 1$, then $\A_t(\Omega) \subset \A_s(\Omega)$.
\end{remark}
\begin{proof}
Let $A\in\A_1(\Omega)$. There exists a decreasing sequence of open sets $\{G_k\}_{k\in\N}$ such that $A\cup G_k$ is open and $\cp_1(G_k,\Omega)\to0$. 

Let $0<s<1$. By Corollary \ref{keycorol}, $\cp_s(G_k,\Omega)\to 0$. Then, $A\in \A_s(\Omega)$.

To prove $\A_t(\Omega) \subset \A_s(\Omega)$ for $0<s<t\le 1$, use definitions of capacity and \cite[Proposition 2.1]{DiNezza-Palatucci-Valdinoci}.
	
\end{proof}


\section*{Acknowledgements}
This paper was partially supported by grants UBACyT 20020130100283BA, CONICET PIP 11220150100032CO and ANPCyT PICT 2012-0153. The author wants to thank Prof. Juli\'an Fern\'andez Bonder for helpful conversations.  

A. Ritorto is a doctoral fellow of CONICET.

\bibliographystyle{amsplain}
\bibliography{biblio}

\end{document}